\documentclass{amsart}
\newtheorem{thm}{Theorem}[section]

\newtheorem{lem}[thm]{Lemma}

\newtheorem{exa}[thm]{Example}

\theoremstyle{definition}
\newtheorem{defn}[thm]{Definition}
\theoremstyle{remark}
\newtheorem{rem}[thm]{Remark}

\numberwithin{equation}{section}

\begin{document}
\vbox{\vskip 3.5cm}
%\vspace{4.2cm}
\begin{center}
{\bf {\large Monotonicity and asymptotic behavior of solutions for Riemann-Liouville fractional differential equations}}
\end{center}
\title{}
\maketitle

\bigskip

\begin{center}
Tao Zhu\\
\bigskip
School of Mathematics and Physics, Nanjing
Institute of Technology,\\
Nanjing, 211100, P. R. China \\
Email: zhutaoyzu@sina.cn\\
\end{center}

%\maketitle
\baselineskip=20pt

\bigskip

\bigskip
\begin{abstract}
In this paper, we first investigate the monotonicity and limit problem of the fractional integral functions. By fixed point theorem and these new results of the fractional integral functions, we present that the Riemann-Liouville fractional differential equations has at least one decreasing solution in $C_{1-\beta}^{+}(0,+\infty)$. The asymptotic behavior of solutions is also discussed under some different conditions. The novelty in this paper is that we investigate the asymptotic behavior of Riemann-Liouville fractional differential equations by the monotonicity of functions. Finally, several
examples are given to illustrate our main results.\\
Keywords: Monotonicity; Asymptotic behavior; Riemann-Liouville fractional derivative;
Fractional differential equations.\\
MSC2020: 26A33; 34A08; 34D05.
\end{abstract}
\bigskip

\section{Introduction}
\bigskip
Fractional differential equations (FDEs) have been of great interest in the past four decades. It is caused both by the intensive development of the theory of fractional calculus itself and by the applications in various sciences and engineerings. Therefore, the theory of fractional differential equations has been developed very quickly. Numerous monographs and research papers have been devoted to the study of fractional differential equations; see the monographs of Samko et al. [18], Podlubny [17], Kilbas et al. [14], Diethelm [8], and the papers of Delbosco and Rodino [7], Lan [16], Choi and Koo [5], Becker et al. [2], Brzdek and Eghbali [3], Webb [21] and Zhu [23].

We note that the investigation of the properties of solutions for fractional differential equations have recently received a lot of attention. For example, Furati and Tatar [9] investigated the asymptotic behavior for solutions of a weighted Cauchy-type nonlinear fractional problem. Chen et al. [4] presented some results for the global attractivity of solutions for fractional differential equations involving Riemann-Liouville derivative. Kassim et al. [12] studied the asymptotic behavior of solutions for a class of nonlinear fractional differential equations involving two Riemann-Liouville fractional derivatives of different orders. Gallegos and Duarte-Mermoud [10] studied the asymptotic behavior of solutions to Riemann-Liouville fractional systems. Zhou [22] studied the attractivity of solutions for fractional evolution equation with almost sectorial operators. Tuan et al. [20] presented some results for existence of global solutions and attractivity for multidimensional fractional differential equations involving Riemann-Liouville derivative. Cong et al. [6] presented some distinct asymptotic properties of solutions to Caputo fractional differential equations. Sousa et al. [19] considered the attractivity of solutions of the fractional differential equation involving the $\psi$-Hilfer fractional derivative. Kassim  and Tatar [13] studied the asymptotic behavior of solutions of fractional differential equations with Hadamard fractional derivatives. Lakshmikantham et al. [15, Lemma 1.7.3] investigated the monotonicity of the solution of Caputo fractional differential equation. In [5], the authors gave an improvement of Lemma 1.7.3 in [15], and proved that the solutions of Caputo fractional differential equation are nondecreasing in $t$ when $f(t,x)=\lambda x$, where $\lambda\geq0$.

Recently, Zhu [24] studied the global attractivity of solutions for the following Riemann-Liouville fractional differential equation
\begin{equation}
\begin{cases}
D_{0^{+}}^{\beta}x(t)=f(t,x(t)),  \qquad t\in(0,+\infty),\\
\lim_{t\rightarrow 0^{+}}t^{1-\beta}x(t)=x_{0}.
\end{cases}
\end{equation}
By weakly singular integral inequalities, Zhu proved that the solutions of the equation (1.1) are globally attractive when
\begin{equation}
|f(t, x)|\leq l(t)|x|^{\mu},
\end{equation}
where $0<\mu\leq1$. Using the Schauder fixed point theorem and generalized Ascoli-Arzela theorem, Zhu [25] investigated the global attractivity of solutions of the equation (1.1) when
\begin{equation}
|f(t, x)|\leq l(t)|x|^{\mu}+k(t),
\end{equation}
where $0<\mu\leq1$. In this paper, our aim is to study the following fractional differential equation
\begin{equation}
\begin{cases}
D_{0^{+}}^{\beta}x(t)=l(t)\phi(x(t))+k(t),  \qquad t\in(0,+\infty),\\
\lim_{t\rightarrow 0^{+}}t^{1-\beta}x(t)=x_{0},
\end{cases}
\end{equation}
where $x_{0}>0$. By fixed point theorem and some new results of the fractional integral functions, we present that the equation (1.4) has at least one decreasing solution in $C_{1-\beta}^{+}(0,+\infty)$, and then we discuss the asymptotic behavior of solution of the equation (1.4) under some different conditions. To the best of our knowledge, there have been few results about the monotonicity of solutions of Riemann-Liouville fractional differential equations. The novelty in this paper is that we investigate the asymptotic behavior of fractional differential equation (1.4) by the monotonicity of functions. We also discuss the asymptotic behavior of solution of the equation (1.1) by the boundedness of solution. As opposed to previous papers, the results presented in this paper are new and they generalize the results in [24, 25].

The outline of this paper is as follows. In Section 2, we introduce some notations, definitions and theorems needed in our proofs. Some new results about the fractional integral functions are also proved in this section. In Section 3, we give some sufficient conditions for the monotonicity and asymptotic behavior of solutions of the equation (1.1) and the equation (1.4). In the last Section, some examples are given to illustrate our main results.

\section{Preliminaries}
In this section, we introduce some notations, definitions and theorems which will be needed later. Some new results about the fractional integral functions are also proved in this section.

For $T>0$, $C^{+}(0,T]$ denotes the set of all nonnegative continuous functions on $(0,T]$.
Let $\alpha\in (0,1)$, we denote
$C^{+}_{\alpha}(0,T]=\{x(t):
x(t)\in C^{+}(0,T]$ and $t^{\alpha}x(t)\in C^{+}[0,T]$$\}$. Let $\|x\|_{\alpha}=\sup_{0<t\leq T}t^{\alpha}x(t)$, then $C^{+}_{\alpha}(0,T]$ endowed with the norm $\|\cdot\|_{\alpha}$ is a Banach space.
We denote $ C_{\alpha}(0,+\infty)=\{x(t): x(t)\in C(0,+\infty)$ and $t^{\alpha}x(t)\in C[0,+\infty)\}$ and $ C^{+}_{\alpha}(0,+\infty)=\{x(t): x(t)\in C^{+}(0,+\infty)$ and $t^{\alpha}x(t)\in C^{+}[0,+\infty)\}$.
$L^{p}[0,T]$ $(p\geq1)$ is the Banach space of all measurable functions with the norm $\|f\|=(\int_{0}^{T}|f(t)|^{p}dt)^{1/p}$. $L^{p}_{Loc}[0,+\infty)$ $(p\geq1)$ is the space of all real valued functions $f$ for which $|f|^{p}$ is Lebesgue integrable over every
bounded subinterval of $[0,+\infty)$.

\begin{defn}\label{defn: 3.1}[18]
Let $\beta\in(0,1)$, the operator $I_{0^{+}}^{\beta}$, defined on $L^{1}[0,T]$ by
$$I_{0^{+}}^{\beta}\rho(t)=\frac{1}{\Gamma(\beta)}\int_{0}^{t}\frac{\rho(s)}{(t-s)^{1-\beta}}ds, $$
is called the Riemann-Liouville fractional integral operator of
order $\beta$.
\end{defn}

\begin{defn}\label{defn: 3.2}[18]
Let $\beta\in(0,1)$, the operator $D_{0^{+}}^{\beta}$, defined by
$$D_{0^{+}}^{\beta}\rho(t)=\frac{d}{dt}I_{0^{+}}^{1-\beta}\rho(t)=\frac{1}{\Gamma(1-\beta)}\frac{d}{dt}\int_{0}^{t}\frac{\rho(s)}{(t-s)^{\beta}}ds,$$
where $I_{0^{+}}^{1-\beta}\rho(t)$ is an absolutely continuous function, is called the Riemann-Liouville fractional differential operator of order $\beta$.
\end{defn}

Now, we will give some lemmas about fractional integral functions which are very useful for the
study of the main results of this paper.
\begin{lem}\label{lem: 2.1}[23]
Let $0<\beta<1$, suppose that $t^{1-\beta}\rho(t)\in L^{p}[0,1]$, where $p>1/\beta$. Then we get
$$\left|\int_{0}^{t}(\frac{t}{t-s})^{1-\beta}\rho(s)ds\right|\leq \frac{2^{1/q}t^{\beta-1/p}}{(q\beta-q+1)^{1/q}}\left(\int_{0}^{t}s^{p(1-\beta)}|\rho(s)|^{p}ds\right)^{1/p}$$
for $t\in[0,1]$, where  $q=\frac{p}{p-1}$.
\end{lem}

\begin{lem}\label{thm:3.1}[25]
Let $0<\beta<1$, suppose that $t^{1-\beta}\rho(t)\in L^{p}[0,1]$, where $p>1/\beta$. Let $$y(t)=\int_{0}^{t}(t-s)^{\beta-1}\rho(s)ds,$$ then $t^{1-\beta}y(t)$ is continuous on $[0,1]$ and $y(t)$ is continuous on $(0,1]$.
\end{lem}

\begin{rem}
In [7], if $\rho(t)\in C_{\alpha}[0,+\infty)$ with $0\leq\alpha<\beta<1$, then $y(t)$ is continuous on $[0,+\infty)$ and $y(0)=0$.
Webb [21] proved that $y(t)$ is continuous on $[0,1]$ when $\rho(t)\in C_{\alpha}[0,1]$, where $0\leq\alpha\leq\beta<1$. Agarwal et al. [1] proved that $y(t)$ is continuous on $[0,1]$ when $\rho(t)\in L^{p}[0,1]$, where $p>1/\beta$. Becker et al. [2] proved that $y(t)$ is continuous on $(0,1]$ when $\rho(t)\in C(0,1]\bigcap L^{1}[0,1]$.
\end{rem}

\begin{lem}\label{thm:3.1}
Let $0<\beta<1$, suppose that $\rho\in C(0,+\infty)\bigcap L^{1}_{Loc}[0,+\infty)$.\\
(1). If $t^{\beta}\rho(t)$ is a nonincreasing function on $(0,+\infty)$, then $$y(t)=\int_{0}^{t}(t-s)^{\beta-1}\rho(s)ds$$
 is a nonincreasing function on $(0,+\infty)$.\\
(2). If $t^{\beta}\rho(t)$ is a nondecreasing function on $(0,+\infty)$, then $y(t)$
 is a nondecreasing function on $(0,+\infty)$.
\end{lem}
\begin{proof}
(1). From Remark 2.5, we know that $y(t)$ is a continuous function on $(0,+\infty)$. We have
$$y(t)=t^{\beta-1}\int_{0}^{t}(1-s/t)^{\beta-1}\rho(s)ds.$$
Let $v=s/t$, then we get
\begin{equation}
y(t)=t^{\beta}\int_{0}^{1}(1-v)^{\beta-1}\rho(tv)dv=\int_{0}^{1}(1-v)^{\beta-1}v^{-\beta}(tv)^{\beta}\rho(tv)dv.
\end{equation}
Since $t^{\beta}\rho(t)$ is a nonincreasing function on $(0,+\infty)$, from (2.1), then we get that $y(t)$ is a nonincreasing function on $(0,+\infty)$.\\
(2). Since $t^{\beta}\rho(t)$ is a nondecreasing function on $(0,+\infty)$, also from (2.1), we get that $y(t)$ is a nondecreasing function on $(0,+\infty)$.
\end{proof}

\begin{lem}\label{thm:3.1}
Let $0<\beta<1$, and suppose that $\rho\in C(0,+\infty)\bigcap L^{1}_{Loc}[0,+\infty)$.
Then $$\lim_{t\rightarrow+\infty}y(t)=\lim_{t\rightarrow+\infty}\int_{0}^{t}(t-s)^{\beta-1}\rho(s)ds=0$$
when $\lim_{t\rightarrow+\infty}t^{\beta}\rho(t)=0$.
\end{lem}
\begin{proof}
We know that $y(t)$ is continuous on $(0,+\infty)$.
For $t>1$, we get
$$y(t)=\int_{0}^{1}(t-s)^{\beta-1}\rho(s)ds+\int_{1}^{t}(t-s)^{\beta-1}\rho(s)ds.$$
We know
$$\left|\int_{0}^{1}(t-s)^{\beta-1}\rho(s)ds\right|\leq(t-1)^{\beta-1}\int_{0}^{1}|\rho(s)|ds,$$
then we get
\begin{equation}
\lim_{t\rightarrow+\infty}\int_{0}^{1}(t-s)^{\beta-1}\rho(s)ds=0.
\end{equation}
Using the inequality (4.4) in [23, Lemma 4.2], let $p>1$ and $2\beta-1<1/p<\beta$, then we have
\begin{equation}
\begin{split}
\left|\int_{1}^{t}(t-s)^{\beta-1}\rho(s)ds\right|&\leq t^{\beta-1}\int_{1}^{t}(\frac{t}{t-s})^{1-\beta}|\rho(s)|ds\\
&\leq\frac{2^{1/q}t^{\beta-1}(t-1)^{\beta-1/p}}{(q\beta-q+1)^{1/q}}\left(\int_{1}^{t}s^{p(1-\beta)}|\rho(s)|^{p}ds\right)^{1/p}\\
&\leq\frac{2^{1/q}t^{2\beta-1-1/p}}{(q\beta-q+1)^{1/q}}\left(\int_{1}^{t}s^{p(1-\beta)}|\rho(s)|^{p}ds\right)^{1/p},\\
\end{split}
\end{equation}
where $q=\frac{p}{p-1}$. In (2.3), if $\int_{1}^{+\infty}s^{p(1-\beta)}|\rho(s)|^{p}ds=+\infty$, then using L'H\^{o}spital's rule, we obtain
\begin{equation}
\begin{split}
\lim_{t\rightarrow+\infty}\frac{\int_{1}^{t}s^{p(1-\beta)}|\rho(s)|^{p}ds}{t^{p+1-2p\beta}}
&=\lim_{t\rightarrow+\infty}\frac{t^{p(1-\beta)}|\rho(t)|^{p}}{(p+1-2p\beta)t^{p-2p\beta}}\\
&=\lim_{t\rightarrow+\infty}\frac{t^{p\beta}|\rho(t)|^{p}}{(p+1-2p\beta)}\\
&=0.\\
\end{split}
\end{equation}
Using (2.2), (2.3) and (2.4), we obtain
$$\lim_{t\rightarrow+\infty}\int_{0}^{t}(t-s)^{\beta-1}\rho(s)ds=0.$$
If $\int_{1}^{+\infty}s^{p(1-\beta)}|\rho(s)|^{p}ds$ is finite, since $2\beta-1-1/p<0$, using (2.2) and (2.3), we can immediately obtain
$$\lim_{t\rightarrow+\infty}\int_{1}^{t}(t-s)^{\beta-1}\rho(s)ds=0$$ and $\lim_{t\rightarrow+\infty}y(t)=0$.
\end{proof}

\begin{lem}\label{thm:3.1}
Let $0<\beta<1$, and suppose that $\rho\in C(0,+\infty)\bigcap L^{1}_{Loc}[0,+\infty)$.
Then $$\lim_{t\rightarrow+\infty}y(t)=\lim_{t\rightarrow+\infty}\int_{0}^{t}(t-s)^{\beta-1}\rho(s)ds=\frac{a\pi}{\sin(\beta\pi)}$$
when $\lim_{t\rightarrow+\infty}t^{\beta}\rho(t)=a$.
\end{lem}
\begin{proof}
We have
$$y(t)=\int_{0}^{t}(t-s)^{\beta-1}[\rho(s)-as^{-\beta}]ds+a\int_{0}^{t}(t-s)^{\beta-1}s^{-\beta}ds.$$
Since $\lim_{t\rightarrow+\infty}t^{\beta}[\rho(t)-at^{-\beta}]=0$, then using Lemma 2.7, we obtain
$$\lim_{t\rightarrow+\infty}\int_{0}^{t}(t-s)^{\beta-1}[\rho(s)-as^{-\beta}]ds=0.$$
Then
\begin{equation}
\begin{split}
\lim_{t\rightarrow+\infty}\int_{0}^{t}(t-s)^{\beta-1}\rho(s)ds
&=\lim_{t\rightarrow+\infty}a\int_{0}^{t}(t-s)^{\beta-1}s^{-\beta}ds\\
&=\frac{a\pi}{\sin(\beta\pi)}.\\
\end{split}
\end{equation}
\end{proof}

\begin{rem}
In [21], Webb proved that the fractional integral function
$$y(t)=\int_{0}^{t}(t-s)^{\beta-1}\rho(s)ds, \qquad \beta>0 $$
is a nondecreasing function on $[0,T]$ when $\rho\in L^{1}[0,T]$ is a nonnegative and nondecreasing function. In Lemma 2.6, let $\beta=1/2$ and  $\rho(t)=t^{-1/3}$, we know that $\rho(t)=t^{-1/3}$ is a decreasing function and $t^{1/2}\rho(t)=t^{1/6}$ is a increasing function on $(0,+\infty)$. Using Lemma 2.6, then we get that $y(t)$ is a increasing function on $(0,+\infty)$. In fact, we know that
$$y(t)=\int_{0}^{t}(t-s)^{-1/2}\rho(s)ds=\int_{0}^{t}(t-s)^{-1/2}s^{-1/3}ds=\frac{\Gamma(1/2)\Gamma(2/3)}{\Gamma(7/6)}t^{\frac{1}{6}}$$ is a increasing function on $(0,+\infty)$. Therefore, our results improve the result obtained by Webb [21].
\end{rem}

\begin{rem}
In Lemma 2.6, if $t^{\beta}\rho(t)$ is a nondecreasing function on $(0,T]$, using the same method as in the proof of Lemma 2.6, we can get that $y(t)$ is a nondecreasing function on $(0,T]$. For example, let
$$y(t)=\int_{0}^{t}\frac{(t-s)^{-1/2}}{1+s}ds,$$
we know that the function $t^{1/2}\rho(t)=\frac{t^{1/2}}{1+t}$ is increasing on $[0,1]$ and decreasing on $[1,+\infty)$, then we obtain that $y(t)$ is a increasing function on $[0,1]$. In fact, let $u=\sqrt{t-s}$, then we get
\begin{equation}
\begin{split}
y(t)&=\int_{0}^{\sqrt{t}}\frac{2}{1+t-u^{2}}du\\
&=\frac{2}{\sqrt{1+t}}\ln(\sqrt{1+t}+\sqrt{t}).\\
\end{split}
\end{equation}
From (2.6), we have
$$\frac{dy}{dt}=\frac{1-\sqrt{\frac{t}{1+t}}\ln(\sqrt{1+t}+\sqrt{t})}{(1+t)\sqrt{t}}.$$
We know that
$$g(t)=\sqrt{\frac{t}{1+t}}\ln(\sqrt{1+t}+\sqrt{t})$$
is a increasing function on $[0,+\infty)$, then
$g(t)<g(1)=\frac{\sqrt{2}}{2}\ln(\sqrt{2}+1)<1$ for $t\in[0,1)$. Therefore, we get $\frac{dy}{dt}>0$ when $t\in(0,1]$, then we get that $y(t)$ is a increasing function on $[0,1]$. In fact, $y(t)$ is a increasing function on $[0,T_{0}]$ first, then decreases on $[T_{0},+\infty)$, where $T_{0}>1$ satisfies the equality $g(T_{0})=1$.

Since $\lim_{t\rightarrow+\infty}\frac{\sqrt{t}}{1+t}=0$, using Lemma 2.7, then we get $\lim_{t\rightarrow+\infty}y(t)=0$. In fact, from (2.6) and using L'H\^{o}spital's rule, we get
$$\lim_{t\rightarrow+\infty}y(t)=\lim_{t\rightarrow+\infty}\frac{2}{\sqrt{t}}=0.$$
\end{rem}

\begin{rem}
Let
$$y(t)=\int_{0}^{t}\frac{(t-s)^{-1/2}}{1+\sqrt{s}}ds,$$
we know $\lim_{t\rightarrow+\infty}t^{\beta}\rho(t)=\lim_{t\rightarrow+\infty}\frac{\sqrt{t}}{1+\sqrt{t}}=1$, then using Lemma 2.8, we get
$$\lim_{t\rightarrow+\infty}y(t)=\lim_{t\rightarrow+\infty}\int_{0}^{t}\frac{(t-s)^{-1/2}}{1+\sqrt{s}}ds=\pi.$$
In fact, we know
$$\frac{1}{\sqrt{t}}-\frac{1}{1+\sqrt{t}}=\frac{1}{\sqrt{t}(1+\sqrt{t})}\leq\frac{1}{\sqrt[3]{t^{2}}}$$
for all $t\in(0,+\infty)$. Then
\begin{equation}
\int_{0}^{t}(t-s)^{-1/2}(s^{-1/2}-s^{-2/3})ds\leq\int_{0}^{t}\frac{(t-s)^{-1/2}}{1+\sqrt{s}}ds\leq\int_{0}^{t}(t-s)^{-1/2}s^{-1/2}ds
\end{equation}
and
\begin{equation}
\pi-\frac{\Gamma(1/2)\Gamma(1/3)}{\Gamma(5/6)}t^{-1/6}\leq\int_{0}^{t}\frac{(t-s)^{-1/2}}{1+\sqrt{s}}ds\leq\pi.
\end{equation}
From (2.8), we get
$$\lim_{t\rightarrow+\infty}\int_{0}^{t}\frac{(t-s)^{-1/2}}{1+\sqrt{s}}ds=\pi.$$
\end{rem}

In [2], Becker et al. obtained the equivalence between the fractional differential equation and the Volterra integral equation.
\begin{thm}\label{thm:3.1}
Let $f(t,x)$ be a function that is continuous on the set
$$\mathbf{B}=\left\{(t,x)\in \mathbb{R}^{2}:0<t\leq T, x\in I \right\},$$ where $I\subseteq \mathbb{R}$ denotes an unbounded interval. Suppose a function $x: (0,T]\rightarrow I$ is continuous and that both $x(t)$ and $f(t,x(t))$ are absolutely integrable on $(0,T]$. Then $x(t)$ satisfies the fractional differential equation (1.1) on $(0,T]$ if and only if it satisfies the following Volterra integral equation
$$x(t)=x_{0}t^{\beta-1}+\frac{1}{\Gamma(\beta)}\int_{0}^{t}(t-s)^{\beta-1}f(s,x(s))ds$$
on $(0,T]$.
\end{thm}

Using fixed point theorem and weakly singular integral inequalities, Zhu [23, Corollary 4.5] obtained the following result of the fractional differential equation (1.1).
\begin{thm}
Let $0<\mu\leq1$ and $p>1/\beta$. Suppose  $f: (0,+\infty)\times\mathbb{R}\rightarrow\mathbb{R}$ is a continuous
function, and there exist nonnegative functions $l(t)$ and $k(t)$ such that
$$|f(t, x)|\leq l(t)|x|^{\mu}+k(t)$$
for all $(t,x)\in (0,+\infty)\times\mathbb{R}$, where $t^{(1-\mu)(1-\beta)}l(t)\in
C(0,+\infty)\bigcap L^{p}_{Loc}[0,+\infty)$ and $t^{1-\beta}k(t)\in
C(0,+\infty)\bigcap L^{p}_{Loc}[0,+\infty)$. Then the fractional differential equation
(1.1) has at least one solution in $C_{1-\beta}(0,+\infty)$.
\end{thm}

Finally, the following Schauder fixed point theorem will be needed in the proof of the main results of
this paper.
\begin{thm}\label{thm:3.1}[11]
If $U$ is a nonempty, closed, convex and bounded subset of a Banach space $E$, and $F:U\rightarrow U$ is completely continuous. Then $F$ has a fixed point in $U$.
\end{thm}

\section{Monotonicity and asymptotic behavior for FDEs}
In this section, we present the monotonicity and asymptotic behavior of solutions for the fractional differential equations. First, we will investigate the equation (1.4) under the following hypotheses:
\begin{enumerate}
\item $\phi(t)$ is a nonnegative, continuous and nondecreasing function on $[0,+\infty)$, and there exists a nonnegative constant $M$ such that $|\phi(t)|\leq M t^{\mu}$ for all $t\in[0,+\infty)$, where $0\leq\mu<1$.
\item $t^{\beta}l(t)$ is a nonnegative, continuous and nonincreasing function on $(0,+\infty)$, and $t^{(1-\mu)(1-\beta)}l(t)\in L^{p}[0,1]$, where $p>1/\beta$.
\item $t^{\beta}k(t)$ is a nonnegative, continuous and nonincreasing function on $(0,+\infty)$, and $t^{1-\beta}k(t)\in L^{p}[0,1]$, where $p>1/\beta$.
\end{enumerate}

\begin{thm}\label{thm:3.1}
Assume that the hypotheses (1)-(3) hold, then the Riemann-Liouville
fractional differential equation (1.4) has at least one decreasing solution in $C^{+}_{1-\beta}(0,+\infty)$.
\end{thm}
\begin{proof}We first show that the fractional differential equation (1.4) has at least one decreasing solution in $C^{+}_{1-\beta}(0,T]$, where $T$ is an arbitrarily  constant. Let us define the operator $F: C^{+}_{1-\beta}(0,T]\rightarrow C^{+}_{1-\beta}(0,T]$ by the formula
\begin{equation}
(Fx)(t)=x_{0}t^{\beta-1}+\frac{1}{\Gamma(\beta)}\int_{0}^{t}(t-s)^{\beta-1}\left[l(s)\phi(x(s))+k(s)\right]ds.
\end{equation}
Since $0\leq\mu<1$, let $R>0$ be a sufficiently large constant such that
\begin{equation}
x_{0}+M_{1}R^{\mu}+M_{2}\leq R,
\end{equation}
where $M_{1}=\frac{2MT^{\beta-1/p}\left(\int_{0}^{T}s^{p(1-\mu)(1-\beta)}l^{p}(s)ds\right)^{1/p}}{\Gamma(\beta)(q\beta-q+1)^{1/q}}$ and $M_{2}=\frac{2T^{\beta-1/p}\left(\int_{0}^{T}s^{p(1-\beta)}k^{p}(s)ds\right)^{1/p}}{\Gamma(\beta)(q\beta-q+1)^{1/q}}$ ($p>1/\beta$ and $q=\frac{p}{p-1}$). We define the subset $U$ of $C^{+}_{1-\beta}(0,T]$ as follows
\begin{equation}
U=\left\{x(t): x(t)\in C^{+}_{1-\beta}(0,T]\mbox{ is a decreasing function and }\|x\|_{1-\beta}\leq R\right\}.
\end{equation}
The set $U$ is nonempty, convex, closed and bounded in $C^{+}_{1-\beta}(0,T]$.

Step 1. We first prove $Fx\in U$ when $x\in U$. Since $x_{0}>0$ and $l(t)$, $k(t)$ and $\phi(t)$ are nonnegative functions, we get that $Fx$ is a nonnegative function on $(0,T]$. We have
\begin{equation}
\begin{split}
t^{1-\beta}\left[l(t)\phi(x(t))+k(t)\right]&\leq t^{1-\beta}\left[Ml(t)x^{\mu}(t)+k(t)\right]\\
&\leq Mt^{(1-\mu)(1-\beta)}l(t)(t^{1-\beta}x(t))^{\mu}+t^{1-\beta}k(t)\\
&\leq MR^{\mu}t^{(1-\mu)(1-\beta)}l(t)+t^{1-\beta}k(t),
\end{split}
\end{equation}
using our assumptions and Lemma 2.4, we get $(Fx)(t)\in C^{+}_{1-\beta}(0,T]$. Since $x(t)$ is a decreasing function on $(0,T]$, then
$t^{\beta}l(t)\phi(x(t))+t^{\beta}k(t)$ is a nonincreasing function on $(0,T]$. By Lemma 2.6 and $x_{0}t^{\beta-1}$ is a decreasing function on $(0,T]$, we get that $(Fx)(t)$ is a decreasing function on $(0,T]$. Let $x\in U$, using Lemma 2.3, then we have
\begin{equation}
\begin{split}
t^{1-\beta}(Fx)(t)
&= x_{0}+\frac{t^{1-\beta}}{\Gamma(\beta)}\int_{0}^{t}(t-s)^{\beta-1}\left[l(s)\phi(x(s))+k(s)\right]ds\\
&\leq x_{0}+\frac{t^{1-\beta}}{\Gamma(\beta)}\int_{0}^{t}(t-s)^{\beta-1}\left[Ml(s)x^{\mu}(s)+k(s)\right]ds\\
&\leq x_{0}+\frac{1}{\Gamma(\beta)}\int_{0}^{t}(\frac{t}{t-s})^{1-\beta}\left[MR^{\mu}s^{\mu(\beta-1)}l(s)+k(s)\right]ds\\
&\leq x_{0}+\frac{2^{1/q}MR^{\mu}t^{\beta-1/p}}{\Gamma(\beta)(q\beta-q+1)^{1/q}}\left(\int_{0}^{t}s^{p(1-\mu)(1-\beta)}l^{p}(s)ds\right)^{1/p}\\
&\quad+\frac{2^{1/q}t^{\beta-1/p}}{\Gamma(\beta)(q\beta-q+1)^{1/q}}\left(\int_{0}^{t}s^{p(1-\beta)}k^{p}(s)ds\right)^{1/p}\\
&\leq x_{0}+M_{1}R^{\mu}+M_{2}\\
&\leq R.\\
\end{split}
\end{equation}
Thus, we have $\|Fx\|_{1-\beta}\leq R$.

Step 2. Now, we prove that the operator $F$ is a compact operator. We only need to prove that $t^{1-\beta}F(U)$ is uniformly bounded and equicontinuous on $[0,T]$. Let $x\in U$, from (3.5), we get $\|Fx\|_{1-\beta}\leq R$. This proves that the set $t^{1-\beta}F(U)$ is bounded. Let $0\leq t_{1}<t_{2}\leq T$, we get
\begin{equation}
\begin{split}
&|t_{2}^{1-\beta}(Fx)(t_{2})-t_{1}^{1-\beta}(Fx)(t_{1})|\\
&\leq\frac{1}{\Gamma(\beta)}\left|\int_{0}^{t_{2}}(\frac{t_{2}}{t_{2}-s})^{1-\beta}l(s)\phi(x(s))ds-\int_{0}^{t_{1}}(\frac{t_{1}}{t_{1}-s})^{1-\beta}l(s)\phi(x(s))ds\right|\\
&\quad+\frac{1}{\Gamma(\beta)}\left|\int_{0}^{t_{2}}(\frac{t_{2}}{t_{2}-s})^{1-\beta}k(s)ds-\int_{0}^{t_{1}}(\frac{t_{1}}{t_{1}-s})^{1-\beta}k(s)ds\right|\\
&\leq\frac{MR^{\mu}}{\Gamma(\beta)}\int_{t_{1}}^{t_{2}}(\frac{t_{2}}{t_{2}-s})^{1-\beta}s^{\mu(\beta-1)}l(s)ds\\
&\quad+\frac{MR^{\mu}}{\Gamma(\beta)}\int_{0}^{t_{1}}\left[(\frac{t_{1}}{t_{1}-s})^{1-\beta}-(\frac{t_{2}}{t_{2}-s})^{1-\beta}\right]s^{\mu(\beta-1)}l(s)ds\\
&\quad+\frac{1}{\Gamma(\beta)}\int_{t_{1}}^{t_{2}}(\frac{t_{2}}{t_{2}-s})^{1-\beta}k(s)ds\\
&\quad+\frac{1}{\Gamma(\beta)}\int_{0}^{t_{1}}\left[(\frac{t_{1}}{t_{1}-s})^{1-\beta}-(\frac{t_{2}}{t_{2}-s})^{1-\beta}\right]k(s)ds.\\
\end{split}
\end{equation}
Using the same procedure as in the proof of the inequality (4.4) and the inequality (4.7) in [23, Lemma 4.2], for any $x\in U$, we get that $|t_{2}^{1-\beta}(Fx)(t_{2})-t_{1}^{1-\beta}(Fx)(t_{1})|$ tends to 0 as $t_{2}\rightarrow t_{1}$.
By Ascoli-Arzela theorem, we know that $FU$ is relatively compact. Therefore, the operator $F$ is a compact operator.

Step 3. We show that $F$ is continuous, that is $x_{n}\rightarrow x$ implies $Fx_{n}\rightarrow F$x. Since $x_{n}\rightarrow x$ in $U$, then we have
\begin{equation}
(\frac{t}{t-s})^{1-\beta}l(s)\left|\phi(x_{n}(s))-\phi(x(s))\right|\leq2MR^{\mu}(\frac{t}{t-s})^{1-\beta}s^{\mu(\beta-1)}l(s).
\end{equation}
Since $t^{(1-\mu)(1-\beta)}l(t)\in L^{p}[0,T]$, from Lemma 2.3, we know $(\frac{t}{t-s})^{1-\beta}s^{\mu(\beta-1)}l(s)\in L^{1}[0,t]$. For every $s\in(0,t)$, we know
\begin{equation}
(\frac{t}{t-s})^{1-\beta}l(s)\phi(x_{n}(s))\rightarrow (\frac{t}{t-s})^{1-\beta}l(s)\phi(x(s))
\end{equation}
as $n\rightarrow+\infty$. From (3.7) and (3.8), using the Lebesgue dominated convergence theorem, then we have
$$\left|t^{1-\beta}(Fx_{n})(t)-t^{1-\beta}(Fx)(t)\right|=\frac{1}{\Gamma(\beta)}\left|\int_{0}^{t}(\frac{t}{t-s})^{1-\beta}[l(s)\phi(x_{n}(s))-l(s)\phi(x(s))]ds\right|\rightarrow 0$$
as $n\rightarrow +\infty$. Therefore, $t^{1-\beta}(Fx_{n})(t)\rightarrow t^{1-\beta}(Fx)(t)$
pointwise on $[0,T]$ as $n\rightarrow +\infty$. With the fact that $F$ is a compact operator, then we get that $\|Fx_{n}-Fx\|_{1-\beta}\rightarrow0$ as $n\rightarrow+\infty$, which implies that the map $F$ is continuous.

Finally, applying the Schauder fixed point theorem 2.14, we can get that there exists a function $x_{1}\in U$ ($x_{1}$ is a decreasing  function and $x_{1}\in U\subset C^{+}_{1-\beta}(0,T]$) such that $Fx_{1}=x_{1}$. From Theorem 2.12, we know that the function $x_{1}$ is also a solution of the fractional differential equation (1.4).

Since T can be chosen arbitrarily larger, then we know that the fractional differential equation (1.4) has at
least one decreasing solution $x\in C^{+}_{1-\beta}(0,+\infty)$. Thus, we complete the proof.
\end{proof}
Now, we study the asymptotic behavior of fractional differential equations (1.4) by the
monotonicity of functions. Since  $t^{\beta}l(t)$ and $t^{\beta}k(t)$ are nonnegative, continuous and nonincreasing functions on $(0,+\infty)$, then we know that $\lim_{t\rightarrow+\infty}t^{\beta}l(t)$ and $\lim_{t\rightarrow+\infty}t^{\beta}k(t)$ exist. We suppose $\lim_{t\rightarrow+\infty}t^{\beta}l(t)=a$ and $\lim_{t\rightarrow+\infty}t^{\beta}k(t)=b$.

\begin{thm}\label{thm:3.1}
Assume that the hypotheses (1)-(3) hold,
suppose there exists a unique nonnegative constant $k$ that satisfies the following equation
\begin{equation}
x=\frac{\pi}{\Gamma(\beta)\sin(\beta\pi)}(a\phi(x)+b).
\end{equation}
Then the the fractional differential equation (1.4) has at
least one decreasing solution $x\in C^{+}_{1-\beta}(0,+\infty)$ and $\lim_{t\rightarrow+\infty}x(t)=k$.
\end{thm}
\begin{proof}
From Theorem 3.1 and Theorem 2.12, we know that the fractional differential equation (1.4) has at
least one decreasing solution $x\in C^{+}_{1-\beta}(0,+\infty)$ that satisfies the following fractional integral equation
\begin{equation}
x(t)=x_{0}t^{\beta-1}+\frac{1}{\Gamma(\beta)}\int_{0}^{t}(t-s)^{\beta-1}[l(s)\phi(x(s))+k(s)]ds,   \qquad t\in(0,+\infty).
\end{equation}
Since $x(t)$ is a nonnegative and decreasing function on $(0,+\infty)$, then we know that $\lim_{t\rightarrow+\infty}x(t)$ exists. Using Lemma 2.8, we get
\begin{equation}
\begin{split}
\lim_{t\rightarrow+\infty}x(t)
&=\lim_{t\rightarrow+\infty}\left[x_{0}t^{\beta-1}+\frac{1}{\Gamma(\beta)}\int_{0}^{t}(t-s)^{\beta-1}\left[l(s)\phi(x(s))+k(s)\right]ds\right]\\
&=\frac{1}{\Gamma(\beta)}\left[\lim_{t\rightarrow+\infty}\int_{0}^{t}(t-s)^{\beta-1}l(s)\phi(x(s))ds+\lim_{t\rightarrow+\infty}\int_{0}^{t}(t-s)^{\beta-1}k(s)ds\right]\\
&=\frac{1}{\Gamma(\beta)}\left[\frac{a\pi}{\sin(\beta\pi)}\phi(\lim_{t\rightarrow+\infty}x(t))+\frac{b\pi}{\sin(\beta\pi)}\right].
\end{split}
\end{equation}
From (3.9), we get $\lim_{t\rightarrow+\infty}x(t)=k$. Thus, we complete the proof.
\end{proof}
From (3.11), we can immediately obtain the following results.
\begin{thm}\label{thm:3.1}
Assume that the hypotheses (1)-(3) hold and $a=0$.
Then the the fractional differential equation (1.4) has at
least one decreasing solution $x\in C^{+}_{1-\beta}(0,+\infty)$ and $\lim_{t\rightarrow+\infty}x(t)=\frac{b\pi}{\Gamma(\beta)\sin(\beta\pi)}$.
\end{thm}

\begin{thm}\label{thm:3.1}
Assume that the hypotheses (1)-(3) hold and $a=b=0$.
Then the the fractional differential equation (1.4) has at
least one decreasing solution $x\in C^{+}_{1-\beta}(0,+\infty)$ and $\lim_{t\rightarrow+\infty}x(t)=0$.
\end{thm}

In Theorem 3.2, if the equation (3.9) has several different nonnegative solutions, we can use the following method to study the asymptotic behavior of the equation (1.4). Let $x^{*}$ be the largest solution of the equation (3.9), then the following equation
\begin{equation}
x=\frac{\pi}{\Gamma(\beta)\sin(\beta\pi)}(a\phi(x+x^{*})+b)-x^{*}.
\end{equation}
has a unique nonnegative solution $x=0$.

\begin{thm}\label{thm:3.1}
Assume that the hypotheses (1)-(3) hold, let $x^{*}$ be the largest solution of the equation (3.9). Then the the fractional differential equation (1.4) has at
least one decreasing solution $x\in C^{+}_{1-\beta}(0,+\infty)$ and $\lim_{t\rightarrow+\infty}x(t)=x^{*}$.
\end{thm}
\begin{proof}
Let
\begin{equation}
(Gz)(t)=x_{0}t^{\beta-1}+\frac{1}{\Gamma(\beta)}\int_{0}^{t}(t-s)^{\beta-1}\left[l(s)\phi(z(s)+x^{*})+k(s)\right]ds-x^{*}.
\end{equation}
For any $z\in C^{+}_{1-\beta}(0,+\infty)$, we get
\begin{equation}
\begin{split}
(Gz)(t)
&\geq x_{0}t^{\beta-1}+\frac{1}{\Gamma(\beta)}\int_{0}^{t}(t-s)^{\beta-1}\left[l(s)\phi(x^{*})+k(s)\right]ds-x^{*}\\
&=x_{0}t^{\beta-1}+\frac{\phi(x^{*})}{\Gamma(\beta)}\int_{0}^{t}(t-s)^{\beta-1}l(s)ds+\frac{1}{\Gamma(\beta)}\int_{0}^{t}(t-s)^{\beta-1}k(s)ds-x^{*}.
\end{split}
\end{equation}
Since $t^{\beta}l(t)$ and $t^{\beta}k(t)$ are nonincreasing functions on $(0,+\infty)$, by Lemma 2.6, we get that the right-hand side of (3.14) is a nonincreasing function on $(0,+\infty)$. For $t>0$, we have
\begin{equation}
\begin{split}
(Gz)(t)
&\geq \lim_{t\rightarrow+\infty}\left[x_{0}t^{\beta-1}+\frac{\phi(x^{*})}{\Gamma(\beta)}\int_{0}^{t}(t-s)^{\beta-1}l(s)ds+\frac{1}{\Gamma(\beta)}\int_{0}^{t}(t-s)^{\beta-1}k(s)ds-x^{*}\right]\\
&=\frac{\pi}{\Gamma(\beta)\sin(\beta\pi)}(a\phi(x^{*})+b)-x^{*}\\
&=0.
\end{split}
\end{equation}
This implies $(Gz)(t)\geq0$ when $z(t)>0$. The remaining proof is similar to that of Theorem 3.1, we can get that the fractional integral equation (3.13) has at least one decreasing solution $z\in C^{+}_{1-\beta}(0,+\infty)$ such that $Gz=z$. Since $z(t)$ is a nonnegative and decreasing function on $(0,+\infty)$, then we know that $\lim_{t\rightarrow+\infty}z(t)$ exists. using Lemma 2.8, we get
\begin{equation}
\begin{split}
\lim_{t\rightarrow+\infty}z(t)
&=\lim_{t\rightarrow+\infty}\left[x_{0}t^{\beta-1}+\frac{1}{\Gamma(\beta)}\int_{0}^{t}(t-s)^{\beta-1}\left[l(s)\phi(z(s)+x^{*})+k(s)\right]ds-x^{*}\right]\\
&=\frac{1}{\Gamma(\beta)}\left[\lim_{t\rightarrow+\infty}\int_{0}^{t}(t-s)^{\beta-1}l(s)\phi(z(s)+x^{*})ds+\lim_{t\rightarrow+\infty}\int_{0}^{t}(t-s)^{\beta-1}k(s)ds\right]-x^{*}\\
&=\frac{1}{\Gamma(\beta)}\left[\frac{a\pi}{\sin(\beta\pi)}\phi(\lim_{t\rightarrow+\infty}z(t)+x^{*})+\frac{b\pi}{\sin(\beta\pi)}\right]-x^{*}.
\end{split}
\end{equation}
From (3.12), we get $\lim_{t\rightarrow+\infty}z(t)=0$. Since $Gz=z$, then we get
\begin{equation}
z(t)+x^{*}=x_{0}t^{\beta-1}+\frac{1}{\Gamma(\beta)}\int_{0}^{t}(t-s)^{\beta-1}\left[l(s)\phi(z(s)+x^{*})+k(s)\right]ds.
\end{equation}
Therefore, from (3.17), we know that the fractional differential equation (1.4) has at
least one decreasing solution $y=z+x^{*}\in C^{+}_{1-\beta}(0,+\infty)$ and $$\lim_{t\rightarrow+\infty}y(t)=\lim_{t\rightarrow+\infty}z(t)+x^{*}=x^{*}.$$ Thus, we complete the proof.
\end{proof}

In the above Theorems, we study the asymptotic
behavior of solution for the fractional differential equation (1.4) when $t^{\beta}l(t)$ and $t^{\beta}k(t)$ are  nonnegative nonincreasing functions on $(0,+\infty)$ (If $t^{\beta}l(t)$ is a nonnegative nonincreasing functions on $(0,+\infty)$, then we can get that $l(t)$ is also a nonnegative nonincreasing functions on $(0,+\infty)$). Now, we study the asymptotic behavior of the fractional differential equation (1.4) when $l(t)$ and $k(t)$ are non-monotonic functions on $(0,+\infty)$.
\begin{thm}\label{thm:3.1}
Let $0<\beta<\gamma<1$ and $0\leq\mu\leq1$. Let $l(t)$ be a nonnegative function with $t^{(1-\mu)(1-\beta)}l(t)\in
C(0,+\infty)\bigcap L^{p}_{Loc}[0,+\infty)$, where $p>1/\beta$, and there exists a nonnegative constant $K$ such that $t^{\gamma}l(t)\leq K$
for $t\in[1,+\infty)$. Let $k(t)$ be a nonnegative function with $t^{1-\beta}k(t)\in
C(0,+\infty)\bigcap L^{p}_{Loc}[0,+\infty)$ and $\lim_{t\rightarrow+\infty}t^{\beta}k(t)=b\geq0$.
Let $\phi(t)$ be a nonnegative and continuous function on $[0,+\infty)$, and there exists a nonnegative constant $M$ such that $|\phi(t)|\leq M t^{\mu}$ for all $t\in[0,+\infty)$. Then the fractional differential equation (1.4) has at least one solution $x(t)\in
C^{+}_{1-\beta}(0,+\infty)$ and $\lim_{t\rightarrow+\infty}x(t)=\frac{b\pi}{\Gamma(\beta)\sin(\beta\pi)}$.
\end{thm}
\begin{proof}
Using Theorem 2.13, we know that the fractional differential equation (1.4) has at
least one solution $x\in C^{+}_{1-\beta}(0,+\infty)$ that satisfies the following fractional integral equation
\begin{equation}
x(t)=x_{0}t^{\beta-1}+\frac{1}{\Gamma(\beta)}\int_{0}^{t}(t-s)^{\beta-1}[l(s)\phi(x(s))+k(s)]ds,   \qquad t\in(0,+\infty).
\end{equation}
Now, we will prove that the solution $x$ of equation (3.18) is a bounded function on $[1,+\infty)$. Let $\beta>1/p>2\beta-1$, using Lemma 2.3, we get
\begin{equation}
\begin{split}
x(t)
&\leq x_{0}t^{\beta-1}+\frac{1}{\Gamma(\beta)}\int_{0}^{t}(t-s)^{\beta-1}k(s)ds+\frac{M}{\Gamma(\beta)}\int_{0}^{t}(t-s)^{\beta-1}l(s)x^{\mu}(s)ds\\
&\leq x_{0}t^{\beta-1}+\frac{1}{\Gamma(\beta)}\int_{0}^{t}(t-s)^{\beta-1}k(s)ds\\
&\quad+\frac{2^{1/q}Mt^{2\beta-1-1/p}}{\Gamma(\beta)(q\beta-q+1)^{1/q}}\left(\int_{0}^{t}s^{p(1-\beta)}l^{p}(s)x^{p\mu}(s)ds\right)^{1/p}.\\
\end{split}
\end{equation}
Since $\int_{0}^{t}(t-s)^{\beta-1}k(s)ds$ is a continuous function on $(0,+\infty)$, and by Lemma 2.8, we get that $\int_{0}^{t}(t-s)^{\beta-1}k(s)ds$ is a continuous and bounded function on $[1,+\infty)$. We denote $$K_{1}=\sup_{1\leq t<+\infty}\frac{1}{\Gamma(\beta)}\int_{0}^{t}(t-s)^{\beta-1}k(s)ds.$$
For $t\in[1,+\infty)$, from (3.19) and $\beta>1/p>2\beta-1$, we have
\begin{equation}
x(t)\leq x_{0}+K_{1}+\frac{2^{1/q}Mt^{2\beta-1-1/p}}{\Gamma(\beta)(q\beta-q+1)^{1/q}}\left(\int_{0}^{t}s^{p(1-\beta)}l^{p}(s)x^{p\mu}(s)ds\right)^{1/p}
\end{equation}
and
\begin{equation}
\begin{split}
x^{p}(t)&\leq2^{p-1}(x_{0}+K_{1})^{p}+\frac{4^{p-1}M^{p}t^{2p\beta-p-1}}{\Gamma^{p}(\beta)(q\beta-q+1)^{p/q}}\int_{0}^{t}s^{p(1-\beta)}l^{p}(s)x^{p\mu}(s)ds\\
&=2^{p-1}(x_{0}+K_{1})^{p}+\frac{4^{p-1}M^{p}t^{2p\beta-p-1}}{\Gamma^{p}(\beta)(q\beta-q+1)^{p/q}}\int_{0}^{1}s^{p(1-\beta)}l^{p}(s)x^{p\mu}(s)ds\\
&\quad+\frac{4^{p-1}M^{p}t^{2p\beta-p-1}}{\Gamma^{p}(\beta)(q\beta-q+1)^{p/q}}\int_{1}^{t}s^{p(1-\beta)}l^{p}(s)x^{p\mu}(s)ds\\
&\leq2^{p-1}(x_{0}+K_{1})^{p}+\frac{4^{p-1}M^{p}R_{1}^{p\mu}}{\Gamma^{p}(\beta)(q\beta-q+1)^{p/q}}\int_{0}^{1}s^{p(1-\mu)(1-\beta)}l^{p}(s)ds\\
&\quad+\frac{4^{p-1}M^{p}}{\Gamma^{p}(\beta)(q\beta-q+1)^{p/q}}\int_{1}^{t}s^{p\beta-1}l^{p}(s)x^{p\mu}(s)ds\\
&=m_{1}+m_{2}\int_{1}^{t}s^{p\beta-1}l^{p}(s)x^{p\mu}(s)ds,\\
\end{split}
\end{equation}
where $R_{1}=\sup_{0<t\leq1}t^{1-\beta}x(t)$. Let $w(t)=x^{p}(t)$, from (3.21), we get
\begin{equation}
w(t)\leq m_{1}+m_{2}\int_{1}^{t}s^{p\beta-1}l^{p}(s)w^{\mu}(s)ds\leq m_{1}+m_{2}\int_{1}^{t}s^{p\beta-1}l^{p}(s)(w(s)+1)ds.
\end{equation}
Using Gronwall inequality, we have
\begin{equation}
w(t)\leq (m_{1}+1)\exp\left(\int_{1}^{t}m_{2}s^{p\beta-1}l^{p}(s)ds\right), \qquad t\in[1,+\infty).
\end{equation}
Since $t^{\gamma}l(t)\leq K$ for $t\in[1,+\infty)$, then we get that $\int_{1}^{+\infty}s^{p\beta-1}l^{p}(s)ds$ is convergent. Then we obtain
\begin{equation}
x^{p}(t)=w(t)\leq (m_{1}+1)\exp\left(\int_{1}^{+\infty}m_{2}s^{p\beta-1}l^{p}(s)ds\right), \qquad t\in[1,\infty).
\end{equation}
Thus, we obtain that $x(t)$ is a bounded function on $[1,+\infty)$, and we denote $R_{2}=\sup_{1\leq t<+\infty}x(t)$.
From (3.18), since $\lim_{t\rightarrow+\infty}t^{\beta}l(t)x^{\mu}(t)\leq\lim_{t\rightarrow+\infty}R_{2}^{\mu}t^{\beta}l(t)=0$ and using Lemma 2.7 and Lemma 2.8, we get
\begin{equation}
\begin{split}
\frac{b\pi}{\Gamma(\beta)\sin(\beta\pi)}&=\lim_{t\rightarrow+\infty}\frac{1}{\Gamma(\beta)}\int_{0}^{t}(t-s)^{\beta-1}k(s)ds\\
&\leq\lim_{t\rightarrow+\infty}x(t)\\
&=\lim_{t\rightarrow+\infty}\left[x_{0}t^{\beta-1}+\frac{1}{\Gamma(\beta)}\int_{0}^{t}(t-s)^{\beta-1}[l(s)\phi(x(s))+k(s)]ds\right]\\
&\leq\frac{M}{\Gamma(\beta)}\lim_{t\rightarrow+\infty}\int_{0}^{t}(t-s)^{\beta-1}l(s)x^{\mu}(s)ds+\frac{b\pi}{\Gamma(\beta)\sin(\beta\pi)}\\
&=\frac{b\pi}{\Gamma(\beta)\sin(\beta\pi)}.\\
\end{split}
\end{equation}
Therefore, the fractional differential equation (1.4) has at least one solution $x(t)\in
C^{+}_{1-\beta}(0,+\infty)$ and $\lim_{t\rightarrow+\infty}x(t)=\frac{b\pi}{\Gamma(\beta)\sin(\beta\pi)}$.
\end{proof}
In the following content, we will investigate the  the asymptotic behavior
of the fractional differential equation (1.1).
\begin{thm}\label{thm:3.1}
Let $0<\beta<\gamma<1$ and $0\leq\mu\leq1$. Suppose that $t^{(1-\mu)(1-\beta)}l(t), t^{(1-\mu)(1-\beta)}l_{1}(t)\in
C(0,+\infty)\bigcap L^{p}_{Loc}[0,+\infty)$, where $p>1/\beta$, and there exists a nonnegative constant $K$ such that
$t^{\gamma}|l(t)|\leq K$ and $t^{\gamma}|l_{1}(t)|\leq K$ for all $t\in [1,+\infty)$. Suppose that $t^{1-\beta}k(t), t^{1-\beta}k_{1}(t)\in
C(0,+\infty)\bigcap L^{p}_{Loc}[0,+\infty)$, and $\lim_{t\rightarrow+\infty}t^{\beta}k(t)=\lim_{t\rightarrow+\infty}t^{\beta}k_{1}(t)=b$. Suppose $f: (0,+\infty)\times\mathbb{R}\rightarrow\mathbb{R}$ is a continuous
function and
\begin{equation}
l_{1}(t)x^{\mu}+k_{1}(t)\leq f(t, x)\leq l(t)x^{\mu}+k(t)
\end{equation}
for all $(t,x)\in (0,+\infty)\times\mathbb{R}$. Then the fractional differential equation
(1.1) has at least one solution $x(t)\in
C_{1-\beta}(0,+\infty)$ and $\lim_{t\rightarrow+\infty}x(t)=\frac{b\pi}{\Gamma(\beta)\sin(\beta\pi)}$.
\end{thm}
\begin{proof}From (3.26), we get
$$|f(t, x)|\leq(|l(t)|+|l_{1}(t)|)|x|^{\mu}+|k(t)|+|k_{1}(t)|$$
for all $(t,x)\in (0,+\infty)\times\mathbb{R}$. Using Theorem 2.13, we know that the fractional differential equation (1.1) has at
least one solution $x\in C_{1-\beta}(0,+\infty)$ that satisfies the following fractional integral equation
\begin{equation}
x(t)=x_{0}t^{\beta-1}+\frac{1}{\Gamma(\beta)}\int_{0}^{t}(t-s)^{\beta-1}f(s,x(s))ds   \qquad t\in(0,+\infty).
\end{equation}
Using the same procedure as in the proof of the Theorem 3.6, we can get that the solution $x$ of the equation (3.27) is a bounded function on $[1,+\infty)$, and $\lim_{t\rightarrow+\infty}t^{\beta}l(t)x^{\mu}(t)ds=\lim_{t\rightarrow+\infty}t^{\beta}l_{1}(t)x^{\mu}(t)ds=0$. Then we have
\begin{equation}
\begin{split}
\frac{b\pi}{\Gamma(\beta)\sin(\beta\pi)}&=\lim_{t\rightarrow+\infty}\left[x_{0}t^{\beta-1}+\frac{1}{\Gamma(\beta)}\int_{0}^{t}(t-s)^{\beta-1}[l_{1}(s)x^{\mu}(s)+k_{1}(s)]ds\right]\\
&\leq\lim_{t\rightarrow+\infty}x(t)\\
&\leq\lim_{t\rightarrow+\infty}\left[x_{0}t^{\beta-1}+\frac{1}{\Gamma(\beta)}\int_{0}^{t}(t-s)^{\beta-1}[l(s)x^{\mu}(s)+k(s)]ds\right]\\
&=\frac{1}{\Gamma(\beta)}\lim_{t\rightarrow+\infty}\int_{0}^{t}(t-s)^{\beta-1}l(s)x^{\mu}(s)ds+\frac{b\pi}{\Gamma(\beta)\sin(\beta\pi)}\\
&=\frac{b\pi}{\Gamma(\beta)\sin(\beta\pi)}.\\
\end{split}
\end{equation}
Therefore, the fractional differential equation (1.1) has at least one solution $x(t)\in
C_{1-\beta}(0,+\infty)$ and $\lim_{t\rightarrow+\infty}x(t)=\frac{b\pi}{\Gamma(\beta)\sin(\beta\pi)}$.
\end{proof}

\begin{thm}\label{thm:3.1}
Let $0<\beta<\gamma<1$ and $0\leq\mu\leq1$. Let $l(t)$ be a nonnegative function with $t^{(1-\mu)(1-\beta)}l(t)\in
C(0,+\infty)\bigcap L^{p}_{Loc}[0,+\infty)$, and $k(t)$ be a nonnegative function with $t^{1-\beta}k(t)\in
C(0,+\infty)\bigcap L^{p}_{Loc}[0,+\infty)$, where $p>1/\beta$. There exists a nonnegative constant $K$ such that
$t^{\gamma}l(t)\leq K$ and $t^{\gamma}k(t)\leq K$
for all $t\in [1,+\infty)$.
Suppose $f: (0,+\infty)\times\mathbb{R}\rightarrow\mathbb{R}$ is a continuous
function and
\begin{equation}
|f(t, x)|\leq l(t)|x|^{\mu}+k(t)
\end{equation}
for all $(t,x)\in (0,+\infty)\times\mathbb{R}$. Then the fractional differential equation
(1.1) has at least one solution $x(t)\in
C_{1-\beta}(0,+\infty)$ and $\lim_{t\rightarrow+\infty}x(t)=0$.
\end{thm}
\begin{proof}
Using the same proof as in the Theorem 3.7, we can get that the fractional differential equation
(1.1) has at least one solution $x(t)\in
C_{1-\beta}(0,+\infty)$ and $x(t)$ is a bounded function on $[1,+\infty)$. Since $t^{\gamma}l(t)\leq K$ and $t^{\gamma}k(t)\leq K$
for all $t\in [1,+\infty)$, then we get
\begin{equation}
\begin{split}
\lim_{t\rightarrow+\infty}|x(t)|&\leq\lim_{t\rightarrow+\infty}\frac{1}{\Gamma(\beta)}\int_{0}^{t}(t-s)^{\beta-1}[l(s)|x(s)|^{\mu}+k(s)]ds\\
&=0.\\
\end{split}
\end{equation}
Therefore, the proof is complete.
\end{proof}

\begin{rem}
In Theorem 3.1, we prove that the fractional differential equation
(1.4) has at least one decreasing solution when $0\leq\mu<1$. In hypothesis (2) of Theorem 3.1, we require that the function $t^{\beta}l(t)$ is a nonincreasing function on $(0,+\infty)$, then we have $t^{\beta}l(t)\geq l(1)$ when $t\in(0,1]$, and
\begin{equation}
\int_{0}^{1}l^{p}(s)ds\geq\int_{0}^{1}l^{p}(1)s^{-p\beta}ds.
\end{equation}
Since $p\beta>1$, then $\int_{0}^{1}l^{p}(s)ds$ is divergent. In hypothesis (2), we also require $t^{(1-\mu)(1-\beta)}l(t)\in L^{p}[0,1]$ when $p>1/\beta$. There is a contradiction with (3.31) when $\mu=1$. Therefore, we can not use the same method in Theorem 3.1 to study the monotonicity of fractional differential equation (1.4) when $\mu=1$.
\end{rem}

\begin{rem}
In Theorem 3.2, let
\begin{equation}
g(x)=x-\frac{\pi}{\Gamma(\beta)\sin(\beta\pi)}(a\phi(x)+b).
\end{equation}
We know \begin{equation}
g(0)=-\frac{\pi}{\Gamma(\beta)\sin(\beta\pi)}(a\phi(0)+b)\leq0
\end{equation}
and
\begin{equation}
g(x)\geq x-\frac{\pi}{\Gamma(\beta)\sin(\beta\pi)}(aMx^{\mu}+b).
\end{equation}
From (3.34), we get $\lim_{x\rightarrow+\infty}g(x)=+\infty$ for $0\leq\mu<1$. Therefore, there has at least one nonnegative solution $x$ such that $g(x)=0$, that is, $x$ is a nonnegative solution of the equation (3.9).
\end{rem}

\begin{rem}
By weakly singular integral inequalities, Zhu [24, Theorem 4.1 and Theorem 4.2] proved that the solution of fractional differential equation
(1.1) is attractive when $|f(t, x)|\leq l(t)|x|^{\mu}$. Using the Schauder fixed point theorem and generalized Ascoli-Arzela theorem, Zhu [25, Theorem 3.1] proved that the fractional differential equation (1.1) has at least one attractive solution when $|f(t, x)|\leq l(t)|x|^{\mu}+k(t)$. In this paper, we propose a new method to study the asymptotic behavior of the fractional differential equation (1.4). In Theorem 3.1 and Theorem 3.2, we get that the fractional differential equation (1.4) has at least one nonnegative decreasing solution and $\lim_{t\rightarrow+\infty}x(t)=k$. In particular, we get $\lim_{t\rightarrow+\infty}x(t)=0$ when $a=b=0$. In Theorem 3.6, we also study the asymptotic behavior of the fractional differential equation (1.4) when $l(t)$ and $k(t)$ are non-monotonic functions. In Theorem 3.7 and Theorem 3.8, we study the asymptotic behavior of the fractional differential equation (1.1). Especially, the constant $b$ could be a negative in Theorem 3.7.

Under the assumptions of theorem 3.8, Zhu [25, Theorem 3.1] obtained that the fractional differential equation (1.1) has at least one globally attractive solution (that is $\lim_{t\rightarrow+\infty}x(t)=0$). In the above Theorems, we study the monotonicity and asymptotic behavior of solutions for the fractional differential equations under some different assumptions. Therefore, our results in this paper generalize the results obtained in [24, 25].
\end{rem}

\section{Applications}
\begin{exa}
Consider the following Riemann-Liouville fractional differential equation
\begin{equation}
\begin{cases}
D_{0^{+}}^{1/2}x(t)=(t^{-3/4}+t^{-1/2})\frac{x(t)+1}{x(t)+2},\qquad t\in(0,+\infty),\\
\lim_{t\rightarrow 0^{+}}t^{1/2}x(t)=1.\\
\end{cases}
\end{equation}
\end{exa}
Since $\phi(t)=\frac{t+1}{t+2}\leq1$, let $\mu=0$ in hypothesis (2), then we get that $t^{1/2}l(t)=1+t^{-1/4}$ is a nonincreasing function on $(0,+\infty)$ and $t^{1/2}l(t)\in L^{p}[0,1]$, where $2<p<4$. From Theorem 3.2, we know that the equation (4.1)
has at least one decreasing solution $x\in C^{+}_{1/2}(0,+\infty)$ and $\lim_{t\rightarrow+\infty}x(t)=\frac{\sqrt{\pi}+\sqrt{4+\pi}-2}{2}$.

\begin{exa}
Consider the following Riemann-Liouville fractional differential equation
\begin{equation}
\begin{cases}
D_{0^{+}}^{1/2}x(t)=t^{-\gamma}\ln(1+x^{\mu}(t)),\qquad t\in(0,+\infty),\\
\lim_{t\rightarrow 0^{+}}t^{1/2}x(t)=1,\\
\end{cases}
\end{equation}
where $0<\mu<1$ and $1/2<\gamma<1-\mu/2$.
\end{exa}
From Theorem 3.4, we know that the equation (4.2)
has at least one decreasing solution $x\in C^{+}_{1/2}(0,+\infty)$ and $\lim_{t\rightarrow+\infty}x(t)=0$.

\begin{exa}
Consider the following Riemann-Liouville fractional differential equation
\begin{equation}
\begin{cases}
D_{0^{+}}^{1/2}x(t)=t^{-1/2}\sqrt[3]{x(t)},\qquad t\in(0,+\infty),\\
\lim_{t\rightarrow 0^{+}}t^{1/2}x(t)=1.\\
\end{cases}
\end{equation}
\end{exa}
We know that the equation
$$x=\sqrt{\pi}\sqrt[3]{x}$$ has two different solutions $x_{1}=0$ and $x_{2}=\pi^{3/4}$. Using Theorem 3.5, we know that the fractional differential equation (4.3) has at least one decreasing solution $x\in C^{+}_{1/2}(0,+\infty)$ and $\lim_{t\rightarrow+\infty}x(t)=\pi^{3/4}$.

\begin{exa}
Consider the following Riemann-Liouville fractional differential equation
\begin{equation}
\begin{cases}
D_{0^{+}}^{1/2}x(t)=\frac{\sqrt{x(t)}+\sqrt{t}}{1+t},\qquad t\in(0,+\infty),\\
\lim_{t\rightarrow 0^{+}}t^{1/2}x(t)=1.\\
\end{cases}
\end{equation}
\end{exa}
From Theorem 3.6, we know that the equation (4.4)
has at least one solution $x\in C^{+}_{1/2}(0,+\infty)$ and $\lim_{t\rightarrow+\infty}x(t)=\sqrt{\pi}$.

\begin{exa}
Consider the following Riemann-Liouville fractional differential equation
\begin{equation}
\begin{cases}
D_{0^{+}}^{1/2}x(t)=\frac{\sqrt{x(t)+t}}{1+t},\qquad t\in(0,+\infty),\\
\lim_{t\rightarrow 0^{+}}t^{1/2}x(t)=1.\\
\end{cases}
\end{equation}
\end{exa}
Since
$$\frac{\sqrt{t}}{1+t}\leq\frac{\sqrt{x+t}}{1+t}\leq\frac{\sqrt{x}+\sqrt{t}}{1+t},$$
for all $t, x\in[0,+\infty)$.
From Theorem 3.7, we know that the equation (4.5)
has at least one solution $x\in C^{+}_{1/2}(0,+\infty)$ and $\lim_{t\rightarrow+\infty}x(t)=\sqrt{\pi}$.

\begin{exa}
Consider the following Riemann-Liouville fractional differential equation
\begin{equation}
\begin{cases}
D_{0^{+}}^{1/2}x(t)=t^{-2/3}\sqrt[3]{x(t)}-t^{-1/2},\qquad t\in(0,+\infty),\\
\lim_{t\rightarrow 0^{+}}t^{1/2}x(t)=1.\\
\end{cases}
\end{equation}
\end{exa}
From Theorem 3.7, we know that the equation (4.6)
has at least one solution $x\in C_{1/2}(0,+\infty)$ and $\lim_{t\rightarrow+\infty}x(t)=-\sqrt{\pi}$.

\bigskip

\centerline{Acknowledgements}
I thank the referee for pointing out some useful references and for valuable comments which led to improvements of the paper.

\centerline{Declarations}
Conflict of interest: The authors declare that they have no conflict of interest.

\end{document}